\documentclass{article}
\usepackage{amsmath,amssymb,amsthm,yfonts}
  \usepackage[usenames,dvipsnames]{pstricks}

\usepackage{epsfig}
 \usepackage{pst-grad} 
 \usepackage{pst-plot} 
 \usepackage[space]{grffile} 
 \usepackage{etoolbox} 
\usepackage{tikz-cd}

\title{Complexity lower bound for typical right triangular billiards}
\author{ Dmitri Scheglov}

\theoremstyle{plain}
\newtheorem{Lemma}{Lemma}[section]

\newtheorem{thm}{Theorem}

\begin{document}

\maketitle
\setlength{\parindent}{0pt}

\begin{abstract}

We prove that for a typical (in the Lebesgue measure sense) right triangular billiard and for any $\epsilon>0$ the complexity function $P_n$ satisfies: $\limsup \frac{P_n}{n^{\frac{2}{\sqrt{3}}-\epsilon}}>0$
\end{abstract}
\section{Overview}

A generalized diagonal of a polygonal billiard is an orbit which connects two
vertices of the polygon. One of the characteristics of billiard dynamics is the complexity function
$P_n$, which counts the number of generalized diagonals of length no greater than $n$. Here
by length we mean the number of reflections. See also [1], [4], [5], [6] for other notions of billiard complexity, related to our definition.
\

\

Katok [9] proved the subexponential estimate.
\

\

\begin{thm} (Katok, 1987). For any polygon: $\lim \frac{ln(P_n)}{n}=0$
\end{thm}
\

Masur [11], [12] using Teichmuller theory proved quadratic estimates for any
rational-angled polygon:
\begin{thm} (Masur, 1990). For any polygon with angles in $\pi \mathbb{Q}$ there are constants $C_1, C_2 > 0$ such that: $C_1 · n^2<P_n<C_2 n^2$
\end{thm}
\

Due to the difficulty of the further progress in non-rational case, Katok formulated the following
question, which he included in his list “Five most resistant problems in dynamics”[10]:
\

\

\textit{Find the upper and lower estimates for} $P_n$.
\

\

The key difficulty here is the luck of structure for general polygonal billiards. In the rational case the billiard flow reduces to the geodesic flow on the translation surface which allows its study by methods of Teichmuller theory. However in irrational case the billiard is only equivalent to the geodesic flow on the flat sphere with finite number of conical singularities, with infinite holonomy group. For the zero measure set of billiards with extremely well-approximated angles ergodicity results were obtained by Katok [8] and then, more explicitly by Vorobets[15]. Recently Chaika and Forni [2] gave an example of a weak-mixing billiard, essentially using rational methods.
\

\

Regarding billiard complexity, in works [13], [14] we developed a bootstrapping and indexed partition technique allowing in its strongest form the stretched exponential estimate for typical triangular billiards.
\

\begin{thm}(S, 2012) For a typical (in the Lebesgue measure sense) triangular billiard and any $\epsilon>0$ the following estimate takes place: $P_n<Ce^{n^{\sqrt{3}-1+\epsilon}}$.
\end{thm}
\

\begin{thm}(S, 2020) For a typical (in the Lebesgue measure sense) triangular billiard and any $\epsilon>0$ the following estimate takes place: $P_n<Ce^{n^{\epsilon}}$.
\end{thm}
\

In the current paper we find a \textit{lower} estimate on $P_n$ for right triangular billiards using indexed partitions, however the geometric idea differs from that of [14]. Informally speaking  the key idea of [14] was based on some delicate interplay between the abundance of \textit{long} generalized diagonals and \textit{short} ones. Here we roughly show that relatively small number of generalized diagonals implies the existence of rather short periodic orbits, which in turn produce new \textit{unwanted} generalized diagonals, which gives a contradiction.
The aim of the current paper is to prove the following estimate:

\begin{thm}For a typical (in the Lebesgue measure sense) right triangular billiard and for any $\epsilon>0$ the complexity function $P_n$ satisfies: $\limsup \frac{P_n}{n^{\frac{2}{\sqrt{3}}-\epsilon}}>0$

\end{thm}
\

\section{Development along the trajectory and geodesic flow on the flat punctured sphere}

We remind a useful development construction associated to any
polygonal billiard [7]. We fix a polygon on the plane and consider a time moment
when a particular billiard orbit hits a polygon side. Then instead of reflecting
the orbit we continue it as a straight line and reflect the polygon along the
line. As we continue this process indefinitely the sequence of polygons obtained
this way is called development of the polygon along the orbit. See fig. $1$, showing triangle development.
\

\
\begin{figure}
\psscalebox{1.0 1.0} 
{
\begin{pspicture}(0,-6.554861)(14.78236,3.0313833)
\pspolygon[linecolor=black, linewidth=0.02](2.1789713,-5.581379)(4.3850155,-1.5864611)(7.1300006,-2.409296)
\pspolygon[linecolor=black, linewidth=0.02](2.2026267,-5.586878)(4.402627,-1.5668777)(2.2226267,0.29312226)
\psline[linecolor=black, linewidth=0.03, arrowsize=0.05291667cm 2.0,arrowlength=1.4,arrowinset=0.0]{->}(0.3426268,-2.1668777)(14.742627,-2.1868777)
\pspolygon[linecolor=black, linewidth=0.02](2.1826267,-5.566878)(0.0226268,-1.5468777)(2.2026267,0.31312224)
\pspolygon[linecolor=black, linewidth=0.02](4.7995067,2.9812706)(4.3936834,-1.5833442)(7.136144,-2.4145544)
\pspolygon[linecolor=black, linewidth=0.02](4.8162565,2.978476)(8.409924,0.1657267)(7.161407,-2.4136543)
\pspolygon[linecolor=black, linewidth=0.02](12.866853,-0.9558881)(8.420342,0.15269212)(7.1740117,-2.4277458)
\pspolygon[linecolor=black, linewidth=0.02](12.865252,-0.9518255)(9.523942,-4.060145)(7.1721635,-2.4227266)
\pspolygon[linecolor=black, linewidth=0.02](12.890776,-0.9414156)(9.5456295,-4.0735707)(11.035921,-6.521229)
\pspolygon[linecolor=black, linewidth=0.02](12.895707,-0.93856764)(13.69782,-5.431074)(11.047332,-6.52053)
\end{pspicture}
}
\end{figure}

Fig. 1. Development of triangular billiard along the trajectory.

\

\

Another useful construction, associated with the polygonal billiard is the geodesic flow on the flat sphere with a finite number of conical singularities. Namely we take a copy $B$ of the polygon $A$ and glue them along corresponding sides. The resulting object is a sphere $S$ with a flat metric on it and a finite number of singularities, corresponding to the vertices of $A$. We slightly modify billiard flow $T_t$ on the polygon $A$. Namely, having a billiard orbit $T_t$ in $A$, we continue it in its copy $B$, after it hits the side of $A$. Analogous thing we make with the billiard in $B$. So the resulting flow "jumps" between $A$ and $B$ after each reflection, thus giving a geodesic flow $\overline{T_t}$ on $S$.
\

\

\textbf{Remark 1.} On has to be a bit careful about relation between \textit{periodic} orbits of flows $T_t$ and $\overline{T_t}$. Namely if $T$ is a periodic orbit of $T_t$ of \textit{even} period (meaning an even number of reflections) then it does correspond to the periodic orbit $\overline{T}$ of $\overline{T_t}$, which is not the case for periodic orbits of an odd period. The issue here is that any orbit on $S$ "jumps" between $A$ and $B$, so there has to be an even number of reflections ( or "side-crossings" on $S$) in order for it to return to the "correct" side.
\

\
\begin{figure}
\psscalebox{1.0 1.0} 
{
\begin{pspicture}(0,-5.979097)(14.034969,1.0727091)
\definecolor{colour0}{rgb}{0.0,0.5019608,0.0}
\pspolygon[linecolor=black, linewidth=0.03](0.02895874,-1.3140969)(2.0089588,1.005903)(4.888959,0.285903)(5.508959,-2.414097)(3.2889588,-5.954097)(3.2889588,-5.954097)
\pspolygon[linecolor=black, linewidth=0.03](8.528958,-1.274097)(10.508959,1.045903)(13.388959,0.32590303)(14.008959,-2.3740969)(11.788959,-5.914097)(11.788959,-5.914097)
\psline[linecolor=blue, linewidth=0.03, arrowsize=0.05291667cm 2.0,arrowlength=1.4,arrowinset=0.0]{->}(2.8089588,-3.514097)(5.2689586,-1.374097)
\psline[linecolor=blue, linewidth=0.03, linestyle=dashed, dash=0.17638889cm 0.10583334cm, arrowsize=0.05291667cm 2.0,arrowlength=1.4,arrowinset=0.0]{->}(11.308959,-3.534097)(13.768959,-1.394097)
\psline[linecolor=red, linewidth=0.03, arrowsize=0.05291667cm 2.0,arrowlength=1.4,arrowinset=0.0]{->}(13.748959,-1.374097)(8.8289585,-0.934097)
\rput[bl](2.3289587,-0.21409698){A}
\rput[bl](10.608959,-0.21409698){B}
\psline[linecolor=red, linewidth=0.03, linestyle=dashed, dash=0.17638889cm 0.10583334cm, arrowsize=0.05291667cm 2.0,arrowlength=1.4,arrowinset=0.0]{->}(5.3089585,-1.354097)(0.38895875,-0.914097)
\psline[linecolor=colour0, linewidth=0.03, arrowsize=0.05291667cm 2.0,arrowlength=1.4,arrowinset=0.0]{->}(0.36895874,-0.894097)(0.38895875,-1.754097)
\end{pspicture}
}
\end{figure}

Fig. 2. Two copies of the billiard table $A$ and $B$ are glued together along sides and form flat sphere with conical singularities at vertices.  Blue geodesic trajectory at the side $A$ of the sphere instead of reflecting from the boundary, continues as a red trajectory at the side $B$ of the sphere. Then red trajectory instead of reflecting from the boundary of $B$, continues as a green trajectory on the side $A$ of the sphere.
\

\

\begin{Lemma} Let  $M$ be a $2$-dimensional oriented manifold with a flat metric, $x_0\in M$ and $v_0\in T_{x_0}(M)$. Let us also assume that the geodesic $\gamma_0$ of length $L$ which starts at point $x_0$ in the direction of the vector $v_0$ is closed. Let $x=x(t)$, $t\in[0, 1]$ be a continuous curve on $M$ and $x(0)=x$. Let $v_t$ be a parallel translation of the vector $v_0$ along the curve $x(t)$ and let $\gamma_t$ be a geodesic of length $L$ which starts at the point $x_t$ in the direction $v_t$. Then $\gamma_t$ is closed for $t\in [0, 1]$.

\end{Lemma}

\begin{proof} For any point $y\in M$, a vector $v\in T_y(M)$ and a constant $L>0$ we consider a shift map $S_L$, defined as follows. Consider a small enough disk $U$ centered at $y$ and a point $a\in U$. Let
$u\in T_aM$ be a parallel translate of $v$ and $\gamma_u(s)$, $s\in[0, L]$ is a geodesic of length $L$ from $a$ into direction $u$. Then we define $S_L(a)=\gamma_u(L)$. It is easily seen that $S_L$ is a local isometry.
\

\

Let now $x_t, v_t$, $t\in[0, 1]$ be as in the conditions of the Lemma. Let $C\subseteq[0, 1]$ be such that for $t\in C$ the trajectory $\gamma_t$ of length $L$ is a closed geodesic. As $S_L$ is a local isometry then $C$ is open. But of course $C$ is closed as the limit of closed geodesics of length $L$ is also a closed geodesic of length $L$. As $C$ is non-empty, it implies that $C=[0, 1]$.

\end{proof}

Fig. 3. Reflecting triangle about right angle 3 times, we make an associated rhombus. Any billiard trajectory in a rhombus naturally projects to the triangle billiard trajectory. For any reflection of the trajectory in the rhombus, there are at most 2 more reflections of the triangular trajectory, corresponding to "diagonal-crossings".
\

\

Another useful lemma will be used later on and will allow us to consider a rhombus instead of triangle which turns out to be slightly easier from technical point of view. We first take a right triangle and reflect it 3 times about the right angle to get a rhombus. As fig.3 shows then the billiard in the rhombus naturally "projects' to the billiard in triangle. As one can see, each trajectory of rhombus billiard can cross at most 2 triangle sides between the two consecutive reflections. That means that any generalized diagonal of rhombus billiard of algebraic length $n$ corresponds to the generalized diagonal of a triangle billiard of length no more than $3n$.
\

\

\textbf{Remark 2.}. The generalized diagonal of a rhombus may possibly pass through the center of the rhombus on its way between rhombus vertices. In this case it \textit{still} gives a generalized diagonal for the triangle billiard, only of the smaller algebraic length.
\

\

Summarizing the arguments above, we have the following useful statement. Let $P_n^{\triangle}$ be a complexity function for a right triangle, and $P_n^{\diamondsuit}$ be a complexity function for a corresponding rhombus.

\begin{Lemma} $P_{3n}^{\triangle}\geq P_n^{\diamondsuit}$
\end{Lemma}
\

\

\begin{figure}
\psscalebox{1.0 1.0} 
{
\begin{pspicture}(0,-2.945)(10.46,0.965)
\psdiamond[linecolor=black, linewidth=0.03, dimen=outer](5.23,-0.995)(5.23,1.95)
\psline[linecolor=black, linewidth=0.03](0.02,-1.005)(10.4,-1.005)
\psline[linecolor=black, linewidth=0.03](5.22,-2.925)(5.22,0.955)
\psline[linecolor=blue, linewidth=0.03, arrowsize=0.05291667cm 2.0,arrowlength=1.4,arrowinset=0.0]{->}(1.38,-1.485)(5.66,0.755)
\psline[linecolor=red, linewidth=0.03, arrowsize=0.05291667cm 2.0,arrowlength=1.4,arrowinset=0.0]{->}(5.64,0.755)(6.56,-2.405)
\end{pspicture}
}
\end{figure}

Fix a rhombus vertex and consider the set of trajectories emanating from the vertex through the fixed angular segment $I$, with a fixed algebraic length $n$, such that no trajectory hits another vertex up to $n$ reflections. Such a set is called \textit{angular} $I$-\textit{beam}. As angular $I$-beam does not hit a vertex( except its boundary rays) , the development up to $n$ reflections is the same for all the trajectories inside it.(see Fig.4)
\

\
\begin{figure}
\psscalebox{1.0 1.0} 
{
\begin{pspicture}(0,-5.916918)(13.9420595,1.8355953)
\psdiamond[linecolor=black, linewidth=0.03, dimen=outer, gangle=-84.83724](2.938032,-0.32323158)(1.19,2.95)
\psdiamond[linecolor=black, linewidth=0.03, dimen=outer, gangle=-40.73992](3.9232721,-2.2663481)(1.19,2.95)
\psdiamond[linecolor=black, linewidth=0.03, dimen=outer, gangle=-356.49136](6.0084157,-2.9724472)(1.19,2.95)
\psdiamond[linecolor=black, linewidth=0.03, dimen=outer, gangle=-312.136](7.9975796,-2.0427017)(1.19,2.95)
\psdiamond[linecolor=black, linewidth=0.03, dimen=outer, gangle=-356.0042](9.96758,-1.1072338)(1.19,2.95)
\psdiamond[linecolor=black, linewidth=0.03, dimen=outer, gangle=-39.876133](12.050726,-1.7501189)(1.19,2.95)
\psline[linecolor=blue, linewidth=0.04, arrowsize=0.05291667cm 2.0,arrowlength=1.42,arrowinset=0.0]{->}(0.04090305,-0.5539782)(8.75407,-1.1519033)(8.75407,-1.1519033)
\psline[linecolor=blue, linewidth=0.04, arrowsize=0.05291667cm 2.0,arrowlength=1.42,arrowinset=0.0]{->}(0.05434477,-0.57886434)(12.903196,-2.4951653)
\psline[linecolor=blue, linewidth=0.02, arrowsize=0.05291667cm 2.0,arrowlength=1.42,arrowinset=0.0]{->}(0.035180837,-0.5731421)(12.985305,-2.2900841)
\psline[linecolor=blue, linewidth=0.02, arrowsize=0.05291667cm 2.0,arrowlength=1.42,arrowinset=0.0]{->}(0.021739118,-0.548256)(13.065415,-2.021789)
\psline[linecolor=blue, linewidth=0.02, arrowsize=0.05291667cm 2.0,arrowlength=1.42,arrowinset=0.0]{->}(0.04090305,-0.5539782)(13.164691,-1.7592163)
\psline[linecolor=blue, linewidth=0.02, arrowsize=0.05291667cm 2.0,arrowlength=1.42,arrowinset=0.0]{->}(0.021739118,-0.548256)(13.263966,-1.4966435)
\end{pspicture}
}
\end{figure}

Fig. 4. Trajectories of the angular beam have the same development up to $n$ reflections. On the figure above $n=6$.
\

\

Similar notion is the \textit{parallel} $\mu$-\textit{beam}. Here we take a set of parallel trajectories of width $\mu$, emanating from some segment on the rhombus side and not hitting a vertex up to $n$-reflections. By length of the parallel $\mu$-beam we mean the length of its middle trajectory. As the parallel $\mu$-beam does not hit a vertex, all the trajectories of a beam have the same development up to $n$ reflections.
\

\

\section{Interval partitions}

We consider a rhombus, a fixed vertex and the corresponding angular segment located at the vertex, which we naturally associate with an interval $I = [0, 1]$ using the normalized angular distance on it. Points on the interval then correspond to directions of rays emanating from the vertex. We introduce a useful reduced quantity $Q_n$ as a number of generalized diagonals of algebraic length no greater than $n$ emanating from the fixed vertex .
Now let us create a decreasing sequence of finite indexed partitions $\zeta_n $ of $I$ on subintervals as follows. The intervals of $\zeta_n$ are formed by the points of $I$ corresponding to the generalized diagonals of algebraic length no greater than $n$ and the endpoints of $I$. We define \textsl{the index} of a cutting point of $\zeta_n$ as the algebraic length of the corresponding generalized diagonal.
Immediate observation is that $Q_n$ equals to the number $|\zeta_n|$ of cutting points of $\zeta_n$. By the diameter \textbf{diam}$(\zeta_n)$ we mean the maximal length of the division intervals.

\begin{Lemma} Let $\zeta_n$ be a sequence of indexed partitions of the interval $[0, 1]$ satisfying the following properties:
\

\

1) There exists $\gamma>0$ such that for all $n$ large enough $|\zeta_n|<n^{\gamma+1}$
\

\

2) There exists $c>0$ such that diam$(\zeta_n)< \frac{c}{n}$
\

\

Then there exist constants $P, K>0$ such that for all $n$ large enough there exists an interval $I\in \zeta_n$ such that:
\

\

 1)$|I|>\frac{P}{n^{\gamma+1}}$
 \

 \

 2) Left and right indices of $I$ are greater than $K n^{\frac{1}{1+\gamma}}$

\end{Lemma}

\begin{proof}
 Assuming that $n$ is large enough we consider a division of the interval $[0, 1]$ on the intervals $I_1=[0, \frac{c}{n}]$, $I_2=[\frac{c}{n}, \frac{2c}{n}]$,$\ldots$, $I_k=[\frac{(k-1)c}{n},\frac{kc}{n}]$, $\ldots$,
 $k\in\{1,[\frac{n}{c}]\}$ and one (possibly consisting of one point) short extra-interval
 $[1-[\frac{n}{c}]\frac{c}{n}, 1]$.
By condition 2 we can pick some points $x_k\in \zeta_n\cap I_k$,  $k\in\{1,[\frac{n}{c}]\}$ and form a finite set $F_n=\{x_1, x_2,\ldots, x_k, \ldots, x_{[\frac{n}{c}]}\}$ of cardinality $[\frac{n}{c}]$.
Then we choose a subset $G_n\subseteq F_n$, $G_n=\{x_1, x_3, x_5, x_7,\ldots\}$ which can be easily seen to satisfy the following two conditions for all $n$ large enough:
\

$1)$ The cardinality $|G_n|> \frac{n}{3c}$
\

$2)$ $\frac{c}{n}\leq Diam(G_n) \leq \frac{3c}{n}$.

\

Let us now consider the set of intervals $J_n=\{[x_1, x_3], [x_3, x_5], [x_5, x_7],\ldots \}$. The cardinality $|J_n|>\frac{n}{3c}$ for all $n$ large enough.
\

\

Let us now divide a set $J_n$ on two subsets $X$ and $Y$ characterized by the following properties:
\

\

$1)$ An interval $J=[a, b] \in J_n$ belongs to $X$ if the cardinality $\zeta_n\cap (a, b)< 6cn^{\gamma}$
\

$2)$ An interval $J=[a, b] \in J_n$ belongs to $Y$ if the cardinality $\zeta_n\cap (a, b)\geq 6cn^{\gamma}$
\

\

As the cardinality $|\zeta_n|<n^{\gamma+1}$ it implies that $|Y|6cn^{\gamma}<n^{\gamma+1}$ and so $|Y|<\frac{n}{6c}$. As $X=J_n-Y$ it implies that $|X|>\frac{n}{3c}-\frac{n}{6c}=\frac{n}{6c}$. As each interval $J$ from $X$ has less than $6cn^{\gamma}$ points of $\zeta_n$ it follows that inside $J$ there is an interval $I\in\zeta_n$  of length at least $\frac{1}{6n^{\gamma+1}}$. Let $Z$ be a set of such intervals $I$. Clearly $|Z|\geq\frac{n}{6c}$.

\

Let $R_n\subseteq\zeta_n$ be a set of points, whose indices are less than ${[\frac{n}{24c}]}^{\frac{1}{\gamma+1}}$. Then, by definition, $R_n\subseteq  {\zeta}_{{[\frac{n}{24c}]}^{\frac{1}{\gamma+1}}}$ and so
the cardinality $|R_n|<|{\zeta}_{{[\frac{n}{24c}]}^{\frac{1}{\gamma+1}}}|<{[{[\frac{n}{24c}]}^{\frac{1}{\gamma+1}}]}^{\gamma+1}=\frac{n}{24c}$.
\

It follows that at least $|Z|-\frac{n}{24c}$ left ends of intervals from $Z$ have indices greater than ${[\frac{n}{24c}]}^{\frac{1}{\gamma+1}}$ and the same is true for right ends. This implies that at least $|Z|-\frac{n}{12c}$ intervals from $Z$ have both left and right indices greater than ${[\frac{n}{24c}]}^{\frac{1}{\gamma+1}}$. As $|Z|>\frac{n}{6c}$ the conclusion of the lemma follows.
\end{proof}

\begin{figure}
\psscalebox{1.0 1.0} 
{
\begin{pspicture}(0,-5.539643)(10.022228,0.90750915)
\psdiamond[linecolor=black, linewidth=0.03, dimen=outer, gangle=-10.855984](4.7328897,-2.2885551)(1.44,3.25)
\psdiamond[linecolor=black, linewidth=0.03, dimen=outer, gangle=-59.48238](7.22244,-3.7842927)(1.44,3.25)
\psdiamond[linecolor=black, linewidth=0.03, dimen=outer, gangle=-322.41183](1.9824399,-2.9642928)(1.44,3.25)
\psline[linecolor=blue, linewidth=0.03, arrowsize=0.05291667cm 2.0,arrowlength=1.42,arrowinset=0.0]{->}(4.0676646,-5.156424)(4.287665,-5.156424)
\psline[linecolor=blue, linewidth=0.03, arrowsize=0.05291667cm 2.0,arrowlength=1.42,arrowinset=0.0]{->}(4.0076647,-4.856424)(4.5076647,-4.856424)
\psline[linecolor=blue, linewidth=0.03, arrowsize=0.05291667cm 2.0,arrowlength=1.42,arrowinset=0.0]{->}(3.8476648,-4.356424)(4.887665,-4.316424)
\psline[linecolor=blue, linewidth=0.03, arrowsize=0.05291667cm 2.0,arrowlength=1.42,arrowinset=0.0]{->}(3.6676648,-3.516424)(5.5076647,-3.436424)
\psline[linecolor=blue, linewidth=0.03, arrowsize=0.05291667cm 2.0,arrowlength=1.42,arrowinset=0.0]{->}(3.4876647,-2.676424)(6.1076646,-2.556424)
\psline[linecolor=blue, linewidth=0.03, arrowsize=0.05291667cm 2.0,arrowlength=1.42,arrowinset=0.0]{->}(3.767665,-3.916424)(5.207665,-3.896424)
\psline[linecolor=blue, linewidth=0.03, arrowsize=0.05291667cm 2.0,arrowlength=1.42,arrowinset=0.0]{->}(3.5876648,-3.0964239)(5.767665,-3.016424)
\psline[linecolor=blue, linewidth=0.03, arrowsize=0.05291667cm 2.0,arrowlength=1.42,arrowinset=0.0]{->}(4.5876646,-5.216424)(5.287665,-5.196424)
\psline[linecolor=blue, linewidth=0.03, arrowsize=0.05291667cm 2.0,arrowlength=1.42,arrowinset=0.0]{->}(4.847665,-4.876424)(5.5876646,-4.856424)
\psline[linecolor=blue, linewidth=0.03, arrowsize=0.05291667cm 2.0,arrowlength=1.42,arrowinset=0.0]{->}(5.5676646,-3.8364239)(6.227665,-3.8364239)
\psline[linecolor=blue, linewidth=0.03, arrowsize=0.05291667cm 2.0,arrowlength=1.42,arrowinset=0.0]{->}(5.887665,-3.436424)(6.5276647,-3.436424)
\psline[linecolor=blue, linewidth=0.03, arrowsize=0.05291667cm 2.0,arrowlength=1.42,arrowinset=0.0]{->}(6.1076646,-3.076424)(6.727665,-3.036424)
\psline[linecolor=blue, linewidth=0.03, arrowsize=0.05291667cm 2.0,arrowlength=1.42,arrowinset=0.0]{->}(6.5076647,-2.576424)(7.207665,-2.556424)
\psline[linecolor=blue, linewidth=0.03, arrowsize=0.05291667cm 2.0,arrowlength=1.42,arrowinset=0.0]{->}(5.187665,-4.416424)(5.9876647,-4.396424)
\psline[linecolor=red, linewidth=0.03, arrowsize=0.05291667cm 2.0,arrowlength=1.42,arrowinset=0.0]{->}(3.4476647,-2.476424)(6.0676646,-2.3764238)
\psline[linecolor=red, linewidth=0.03, arrowsize=0.05291667cm 2.0,arrowlength=1.42,arrowinset=0.0]{->}(3.3276649,-2.016424)(5.9676647,-1.8964239)
\psline[linecolor=red, linewidth=0.03, arrowsize=0.05291667cm 2.0,arrowlength=1.42,arrowinset=0.0]{->}(3.8476648,-1.2964239)(5.787665,-1.196424)
\psline[linecolor=red, linewidth=0.03, arrowsize=0.05291667cm 2.0,arrowlength=1.42,arrowinset=0.0]{->}(4.307665,-0.63642395)(5.687665,-0.5964239)
\psline[linecolor=red, linewidth=0.03, arrowsize=0.05291667cm 2.0,arrowlength=1.42,arrowinset=0.0]{->}(4.707665,-0.01642395)(5.5076647,-0.01642395)
\psline[linecolor=red, linewidth=0.03, arrowsize=0.05291667cm 2.0,arrowlength=1.42,arrowinset=0.0]{->}(4.9876647,0.34357604)(5.4676647,0.34357604)
\psline[linecolor=red, linewidth=0.03, arrowsize=0.05291667cm 2.0,arrowlength=1.42,arrowinset=0.0]{->}(0.1676648,-0.99642396)(0.8676648,-0.976424)
\psline[linecolor=red, linewidth=0.03, arrowsize=0.05291667cm 2.0,arrowlength=1.42,arrowinset=0.0]{->}(0.2476648,-1.4364239)(1.0276648,-1.4364239)
\psline[linecolor=red, linewidth=0.03, arrowsize=0.05291667cm 2.0,arrowlength=1.42,arrowinset=0.0]{->}(0.3876648,-1.956424)(1.1476648,-1.956424)
\psline[linecolor=red, linewidth=0.03, arrowsize=0.05291667cm 2.0,arrowlength=1.42,arrowinset=0.0]{->}(0.5276648,-2.476424)(1.3276649,-2.476424)
\psline[linecolor=red, linewidth=0.03, arrowsize=0.05291667cm 2.0,arrowlength=1.42,arrowinset=0.0]{->}(0.6276648,-3.016424)(1.3876648,-2.996424)
\psline[linecolor=red, linewidth=0.03, arrowsize=0.05291667cm 2.0,arrowlength=1.42,arrowinset=0.0]{->}(0.7476648,-3.436424)(1.5276648,-3.456424)
\psline[linecolor=red, linewidth=0.03, arrowsize=0.05291667cm 2.0,arrowlength=1.42,arrowinset=0.0]{->}(0.8276648,-3.796424)(1.5676647,-3.776424)
\psline[linecolor=red, linewidth=0.05](5.307665,0.88357604)(3.3476648,-2.036424)(3.4876647,-2.636424)
\psline[linecolor=blue, linewidth=0.05](3.4876647,-2.676424)(4.1276646,-5.4764237)
\end{pspicture}
}
\end{figure}

Fig.5. Blue part of $X_n$ represents those points of $X_n$ which are mapped to $X_{n+1}$ by the rhombus development map $\mathcal{R}$. Correspondingly red part of $X_n$ is mapped to $X_{n-1}$ by $\mathcal{R}$.

\section{Rhombus development map}
\

Our construction of rhombus development map is partially motivated by [3], where authors have shown the abundance of perpendicular periodic orbits for right triangular billiards. On the coordinate $xy$-plane with counterclockwise orientation we fix some rhombus $R_0$ and arbitrarily choose a pair of opposite vertices $A$ and $B$. Let $\alpha$ be a rhombus angle, adjacent to the vertex $A$. Let us denote $R_n$ the rhombus, obtained by rotation of $R_0$ to the angle $n\alpha$, $n\in\mathbb{Z}$. Assuming that $\alpha$ is irrational, all $R_n$ are different up to a parallel translation. On the boundary of $R_n$ we consider the set $X_n$ of all points $x$, such that the vector $v=(1, 0)$ located at $x$, points inside $R_n$. As the boundary of $R_n$ consists of two pairs of parallel sides, it is clear that $X_n$ is the union of two adjacent sides of $R_n$.
\

\

Let $X=\cup X_n$ and let us define a rhombus development map $\mathcal{R}: X\longrightarrow X$. Picking a point $x\in X_n$ we shoot a billiard trajectory from $x$ along the vector $v=(1, 0)$. Depending on where this trajectory hits the boundary of $R_n$ we get a point $\mathcal{R}(x)$ belonging to $X_{n-1}$ or $X_{n+1}$ up to a parallel translation. Clearly $\mathcal{R}$ is an invertible map, well defined on $X$ except the countable number of points which hit vertices. This countable number of points will make no difference for our considerations, so slightly abusing the notation we will assume that $\mathcal{R}$ is defined on the whole set $X$.(See Fig. 5.)
\

\

Let $\lambda_n$ be a measure on $X_n$ defined as the Lebesgue measure of projection to the vertical direction. Clearly all $\lambda_n$ then define a measure $\lambda$ on $X$ preserved by $\mathcal{R}$.
\

\

Let us now consider a rotation $R_{\alpha}: S^1\longrightarrow S^1$ on irrational angle $\alpha$ on the circle $S^1=\mathbb{R}/\mathbb{Z}$.
For $\mu\in (0, 0.5)$ we define a \textit{hitting function} $L(\mu)$ as follows. $L(\mu)$ is a minimal integer such that the finite trajectory
$\{x, R_{\alpha}x, \ldots, R^{L(\mu)}_{\alpha}x\}$ is $\mu$-dense, meaning that any interval of length $\mu$ has a point of this finite set inside.

\begin{Lemma} For a typical angle $\alpha\in [0, \pi]$ and any $\epsilon>0$ there is a constant $C=C(\alpha, \epsilon)>0$  such that $L(\mu)< \frac{C}{{\mu}^{2+\epsilon}}$.

\end{Lemma}

\begin{proof} Let $q_n$ be a sequence of continued fraction denominators for a number $\alpha$. By classical theorem of Khintchin, we have typically $\lim\limits_{n \to \infty} \frac{\ln(q_n)}{n}=\beta$. It implies that for $n$ large enough: $ e^{n(\beta-\epsilon)}<q_n<  e^{n(\beta+\epsilon)}$. By well-known estimate : $\frac{1}{2q_{n+1}}<\frac{1}{q_{n+1}+q_n}< ||\alpha q_n||<\frac{1}{q_{n+1}}$. Now we look for an integer $n$, such that
$\frac{1}{q_{n+1}}<\mu$. Using the estimate above we have: $e^{(n+1)(\beta+\epsilon)}> q_{n+1}>\frac{1}{\mu}$ which implies that $n=\frac{\ln (\frac{1}{\mu})}{\beta-\epsilon}$ works. Applying $\alpha$-rotation $q_n$ times we have $\frac{1}{2q_{n+1}}< ||\alpha q_n||<\frac{1}{q_{n+1}}<\mu$ which in turn immediately implies that applying $\alpha$-rotation $2q_{n+1}q_n$ times to any point $x\in S^1$ we obtain a $\mu$-dense set.
\

From the considerations above we obtain a final estimate $L(\mu)< 2q_{n+1}q_n< 2e^{(n+1)(\beta+\epsilon)}e^{n(\beta+\epsilon)}<Ce^{2n(\beta+\epsilon)}
<Ce^{2\frac{\beta+\epsilon}{\beta-\epsilon}\ln(\frac{1}{\mu})}<\frac{C}{\mu^{2+\epsilon}}$

\end{proof}

\begin{Lemma} For a typical rhombus billiard and any $\epsilon>0$ there is a constant $C=C(\epsilon)>0$ such that any parallel $\mu$-beam of length greater than $T(\mu)=\frac{C}{{\mu}^{3+\epsilon}}$ has a parallel trajectory, contained inside.

\end{Lemma}
\

\

\textbf{Remark.} The parallel trajectory, contained inside the parallel $\mu$-beam should not necessarily be one of the beam trajectories ,nor should it have its ends on the beam sides. It literally means a segment,
geometrically contained inside the beam as a subset, which represents a periodic trajectory of a rhombus billiard.
\

\

\begin{proof}Taking a parallel $\mu$-beam we assume that it starts from some rhombus $R_0$. Let us call $I\in X_0$ an initial side of the beam. It follows that the development of $R_0$ along the beam produces a sequence $R_{a_n}$ where $\{a_n\}$ is a sequence of integer numbers with the property $|a_{n+1}-a_n|=1$. Let us denote $I_n={\mathcal{R}}^n(I)\in X_n$.
\

 It follows that there exist  positive integers $N_{-}$ and $N_{+}$ such that the gap for passing from $R_{N_{-}}$ to $R_{-N_{-}-1}$ has width less than $\mu$ and the gap for passing from $R_{N_{+}}$ to $N_{+}+1$ position has with less than $\mu$ (see fig.6). From lemma 4.1  it follows that $N_{-}$ and $N_{+}$ both satisfy
 $N_{-},N_{+}<\frac{C}{\mu^{2+\epsilon}}$
\

\begin{figure}
\psscalebox{1.0 1.0} 
{
\begin{pspicture}(0,-5.555309)(8.018796,3.8170772)
\psdiamond[linecolor=black, linewidth=0.03, dimen=outer, gangle=-57.04132](3.4626632,-1.5938071)(1.41,3.96)
\psdiamond[linecolor=black, linewidth=0.03, dimen=outer, gangle=-96.293755](4.082663,-4.153807)(1.41,3.96)
\psline[linecolor=red, linewidth=0.04](0.99266326,-2.6738071)(1.0526633,-2.593807)
\psline[linecolor=red, linewidth=0.06](0.9726633,-2.6738071)(2.6926632,-0.43380708)
\psline[linecolor=red, linewidth=0.06](2.6926632,-0.43380708)(6.732663,0.5261929)
\psline[linecolor=blue, linewidth=0.06](0.15266328,-3.6938071)(0.93266326,-2.6938071)
\psdiamond[linecolor=black, linewidth=0.03, dimen=outer, gangle=-17.778597](1.3426633,0.046192933)(1.41,3.96)
\psline[linecolor=black, linewidth=0.06](0.95266324,-2.6938071)(4.232663,-2.733807)
\end{pspicture}
}
\end{figure}

Fig. 6. Blue part of $X_{N_{+}}$ is mapped to $X_{N_{+}+1}$ by $\mathcal{R}$ and the red part of $X_{N_{+}}$ is mapped to $X_{N_{+}-1}$ by $\mathcal{R}$. So no parallel $\mu$-beam can pass to $X_{N_{+}+1}$ without splitting if the vertical width of the blue gap is less than $\mu$.
\

\

But as the union $\bigcup\limits_{n=-N_{-}}^{N_{+}} X_n$ has a measure $\lambda(\bigcup\limits_{n=-N_{-}}^{N_{+}} X_n)< N_{+}-N_{-}<\frac{2C}{\mu^{2+\epsilon}}$ and as $\mathcal{R}$ preserves $\lambda$ and $\lambda(I)=\mu$ it follows that there exist
$p, q: 1\leq p\neq q\leq \frac{2C}{\mu^{3+\epsilon}} $,  such that $I_p\cap I_q\neq \varnothing$. Choosing a point $x\in I_n\cap I_m$ and connecting $x\in R_{a_n}$ to $x\in R_{a_m}$ we get a periodic trajectory of length no greater than $\frac{2C}{\mu^{3+\epsilon}}$ inside $\mu$-beam (see fig. 7).
\

\

What is important for us is that such a periodic trajectory necessarily has an \textit{even} period. Indeed assume that development from $R_{a_p}$ to $R_{a_q}$ has $u$ clockwise and $v$ counterclockwise $\alpha$-rotations. But then the rhombus $R_{a_q}$ would be obtained from $R_{a_p}$ by rotation to the angle $(u-v)\alpha$ up to a parallel translation. As the rhombi $R_{a_p}$ and $R_{a_q}$ are parallel and $\alpha$ is irrational, it implies that $u=v$, and so the trajectory has an even period.

\

\end{proof}
\

\
\begin{figure}
\psscalebox{0.8 0.8} 
{
\begin{pspicture}(0,-6.358333)(13.25,3.918333)
\definecolor{colour0}{rgb}{0.8,0.2,0.2}
\psdiamond[linecolor=black, linewidth=0.04, dimen=outer](1.46,-2.101667)(1.46,4.24)
\psdiamond[linecolor=black, linewidth=0.04, dimen=outer, gangle=38.0](4.04,-1.2216668)(1.46,4.24)
\psdiamond[linecolor=black, linewidth=0.04, dimen=outer](6.62,-0.32166687)(1.46,4.24)
\psdiamond[linecolor=black, linewidth=0.04, dimen=outer, gangle=-38.0](9.2,-1.2216668)(1.46,4.24)
\psdiamond[linecolor=black, linewidth=0.04, dimen=outer](11.79,-2.1183336)(1.46,4.24)
\psline[linecolor=black, linewidth=0.04](10.32,-2.061667)(0.26,-1.3816669)
\psline[linecolor=black, linewidth=0.04](10.8,-0.7816669)(0.74,-0.10166687)
\psline[linecolor=colour0, linewidth=0.1](0.26,-1.3616669)(0.7,-0.10166687)
\psline[linecolor=colour0, linewidth=0.1](10.32,-2.041667)(10.78,-0.7816669)
\psline[linecolor=colour0, linewidth=0.1](10.34,-2.061667)(10.8,-0.80166686)
\psline[linecolor=blue, linewidth=0.1](0.3,-1.1816669)(10.7,-1.1416669)
\end{pspicture}
}
\end{figure}

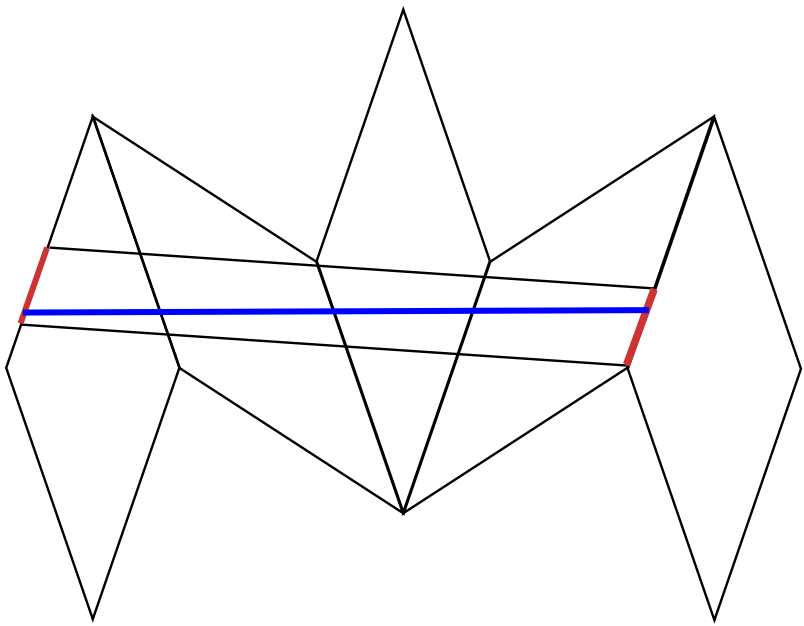
Fig.7. Periodic orbit is inside the $\mu$-beam connecting parallel rhombi $R_{a_p}$ and $R_{a_q}$ as soon as $I_p\cap I_q\neq \varnothing$. ($I_p$ is a left red segment and $I_q$ is a right red segment.)
\

\

\begin{thm} For a typical right triangular billiard and for any $\epsilon>0$ the complexity function $P_n$ satisfies: $\limsup \frac{P_n}{n^{\frac{2}{\sqrt{3}}-\epsilon}}>0$

\end{thm}

\begin{proof} Let us assume that there are positive constants $\gamma<\frac{2}{\sqrt{3}}-1$ and $C$ such that $P_n<Cn^{1+\gamma}$ for all $n$. By reflecting the triangle about the right angle we make a rhombus and by Lemma 2.2 we may assume that the rhombus complexity function also satisfies $P_n<Cn^{1+\gamma}$ ( with different constant $C$). Then we choose any vertex of the rhombus and consider the reduced complexity function for the vertex, which also then satisfies $P_n<Cn^{1+\gamma}$. Consider the sequence of associated indexed partitions for the vertex. Fix $n$ large enough.
\

Then by Lemma 3.1 there is an angular segment $I$ and positive constants $P, K$ ( both independent on $n$) such that:
\

\

$1)$ $|I|> Pn^{-\gamma-1}$
\

\

$2)$ The indices of left and right ends of $I$ are greater than $Kn^{\frac{1}{\gamma+1}}$.
\

\

Taking angular beam corresponding to the segment $I$ we see that the length of the large angular segment, corresponding to the both left and right indices of $I$ has length greater than
 $cn^{\frac{1}{\gamma+1}-\gamma-1}$ for some constant $c$ independent on $n$.
  \
  
  \
  
  Fig.8. Depending on the direction of the periodic trajectory inside the $\mu$-beam, one can drag it either to the original vertex or to one of the two boundary vertices of the interval $I$. As the dragged orbit is also periodic, it implies the existence of a singular point inside the angular beam, which is absurd.
  
  \
  
  \

  Let us pick $\mu=\frac{c}{2}n^{\frac{1}{\gamma+1}-\gamma-1}$ and $\epsilon$ small enough. We are interested to see if we can find a parallel $\mu$-beam of length $T(\mu)=\frac{C}{{\mu}^{3+\epsilon}}$ inside the angular $I$-beam up to the lower index of $I$  for $n$ large enough. Taking into account that the length of the angular $I$-beam is greater then
 $n^{\frac{1}{\gamma+1}}$ we get the sufficient condition:
 \

 \

\begin{figure}
\psscalebox{0.6 0.6} 
{
\begin{pspicture}(0,-8.329899)(23.558989,6.0282817)
\psline[linecolor=black, linewidth=0.04](0.21,1.9190906)(14.37,1.8990905)
\psline[linecolor=black, linewidth=0.04](0.19,1.9390906)(14.03,5.9990907)
\psline[linecolor=black, linewidth=0.04](5.03,1.9390906)(5.91,2.6990905)(8.07,2.6990905)(7.07,1.9190906)
\psline[linecolor=red, linewidth=0.04](5.15,2.0790906)(7.61,2.3590906)
\psline[linecolor=red, linewidth=0.04](0.17,1.9190906)(2.63,2.1990905)
\psline[linecolor=blue, linewidth=0.04, linestyle=dashed, dash=0.17638889cm 0.10583334cm, arrowsize=0.05291667cm 2.0,arrowlength=1.42,arrowinset=0.0]{->}(6.45,2.3390906)(2.27,2.0990906)
\psline[linecolor=black, linewidth=0.04](0.03,-3.4009094)(15.33,0.91909057)
\psline[linecolor=black, linewidth=0.04](0.03,-3.4009094)(15.61,-3.3809094)
\psline[linecolor=black, linewidth=0.04](10.35,-3.3609095)(9.97,-1.8409095)(13.43,-1.8009094)
\psline[linecolor=black, linewidth=0.04](13.41,-1.7609094)(13.93,-3.3409095)
\psline[linecolor=red, linewidth=0.04](10.31,-3.3009095)(13.43,-1.8609095)
\psline[linecolor=red, linewidth=0.04](12.21,-0.5209094)(15.33,0.91909057)
\psline[linecolor=blue, linewidth=0.04, linestyle=dashed, dash=0.17638889cm 0.10583334cm, arrowsize=0.05291667cm 2.0,arrowlength=1.42,arrowinset=0.0]{->}(11.57,-2.6409094)(13.83,0.13909058)
\psdots[linecolor=black, dotsize=0.14](15.31,0.85909057)
\psline[linecolor=black, linewidth=0.04](15.29,0.91909057)(15.63,-3.3609095)
\psdots[linecolor=black, dotsize=0.14](12.21,-0.5209094)
\psdots[linecolor=black, dotsize=0.14](23.49,-2.8209095)
\psdots[linecolor=black, dotsize=0.14](0.17,1.9190906)
\psdots[linecolor=black, dotsize=0.14](2.61,2.2190905)
\psline[linecolor=black, linewidth=0.04](0.01,-8.28091)(15.07,-8.28091)
\psline[linecolor=black, linewidth=0.04](0.03,-8.24091)(14.37,-5.8409095)
\psline[linecolor=black, linewidth=0.04](7.79,-8.260909)(7.99,-7.620909)(10.79,-7.600909)
\psline[linecolor=black, linewidth=0.04](10.77,-7.5809093)(10.61,-8.28091)
\psline[linecolor=red, linewidth=0.04](7.97,-7.6609097)(10.63,-8.16091)
\psdots[linecolor=black, dotsize=0.14](15.05,-8.260909)
\psline[linecolor=black, linewidth=0.04](14.35,-5.8209095)(15.05,-8.220909)
\psline[linecolor=red, linewidth=0.04](12.39,-7.7409096)(15.05,-8.24091)
\psdots[linecolor=black, dotsize=0.14](12.39,-7.7409096)
\psline[linecolor=blue, linewidth=0.04, linestyle=dashed, dash=0.17638889cm 0.10583334cm, arrowsize=0.05291667cm 2.0,arrowlength=1.42,arrowinset=0.0]{->}(8.89,-7.7409096)(13.29,-8.00091)
\end{pspicture}
}
\end{figure}

\

\

 $(-3-\epsilon)(\frac{1}{\gamma+1}-\gamma-1)<\frac{1}{\gamma+1}$.
 \

 \

 Looking for maximal $\gamma$ satisfying inequality above we take a limit case $\epsilon=0$ and get equation: $3(\gamma+1-\frac{1}{\gamma+1})=\frac{1}{\gamma+1}$ which gives a critical value
 ${\gamma}_{crit}=\frac{2}{\sqrt{3}}-1$. As by theorem conditions $\gamma<{\gamma}_{crit}$ then picking $\epsilon$ small enough we will be able to find a parellel $\mu$-beam of length $L(\mu)$ and so to satisfy the condition of the Lemma 4.2.
 \

 \

 Then by Lemma 4.2 we are able to find a periodic trajectory $T$ of even period inside the parallel $\mu$-beam and so inside an angular $I$-beam.
 \

 \

 Let $\alpha\in(-\frac{\pi}{2},\frac{\pi}{2})$ be an angle which $T$ forms with horizontal direction, and $\beta$ be the angle for $I$-beam. We have three cases:
 \

 \

 $1)$ $\alpha\in(-\frac{\pi}{2}, 0)$
 \

 \

 $2)$ $\alpha\in[0, \beta]$
 \

 \

 $3)$ $\alpha\in(\beta, \frac{\pi}{2})$
 \

 \

 Depending on the case we are able to drag the trajectory $T$ inside the angular $I$-beam or to the left cut-point of $I$ or to the right cut-point of $I$ or to the vertex of $I$. By Lemma 2.1 the dragged trajectory is also periodic, as periodic rhombus trajectories of even period correspond to the periodic trajectories on the associated flat sphere (See Remark 1.).  But that would give a contradiction because in each of the three cases we would get a new vertex inside the angular $I$ segment, as the dragged periodic trajectory starts from a vertex and so should end in a vertex.

\end{proof}
\

\end{document}